\documentclass[11pt,reqno]{amsart}
\usepackage[margin=1in]{geometry}
\usepackage[ngerman,french,english]{babel}
\usepackage{amsmath,amssymb,amscd,graphicx, amsthm,
  enumerate, hyperref, tikz} 
  \usetikzlibrary{calc}
\usepackage{rotating}
\usepackage{color,xcolor}

\usepackage{mathrsfs}
\usepackage{mathtools}
\usepackage[all]{xy}
\usepackage{amsmath}
\usepackage{color,xcolor}
\definecolor{darkred}{rgb}{1,0,0} 
\definecolor{darkgreen}{rgb}{0,0.8,0}
\definecolor{darkblue}{rgb}{0,0,1}
\usetikzlibrary{positioning}
  \usepackage{tkz-euclide}
\usetikzlibrary{arrows.meta}

\hypersetup{colorlinks,
linkcolor=darkblue,
filecolor=darkgreen,
urlcolor=darkred,
citecolor=darkgreen}

\numberwithin{equation}{section}

\theoremstyle{plain}

\theoremstyle{plain}
\newtheorem{theorem}{Theorem}
\numberwithin{theorem}{section}
\newtheorem{proposition}[theorem]{Proposition}
\newtheorem{lemma}[theorem]{Lemma}

\newtheorem{corollary}[theorem]{Corollary}

\theoremstyle{definition}
\newtheorem{examplex}[theorem]{Example}

\theoremstyle{definition}
\newtheorem{remark}[theorem]{Remark}
\newtheorem{definition}[theorem]{Definition}
\newtheorem{notation}[theorem]{Notation}

\newcommand{\supp}{\mathrm{supp}}

\newcommand{\interior}[1]{%
  {\kern0pt#1}^{\mathrm{o}}%
}

\newtheoremstyle{named}{}{}{\itshape}{}{\bfseries}{.}{.5em}{\thmnote{#3 }#1}
\theoremstyle{named}
\newtheorem*{namedtheorem}{Theorem}

\DeclareMathOperator{\ord}{ord}

\newcommand{\R}{\mathbb{R}}

\newcommand{\C}{\mathbb{C}}

\title[Non-solvability in the flat category of elliptic operators]{Non-solvability in the flat category of elliptic operators with real analytic coefficients}

\author{Martino Fassina}
\address{
Fakult\"at f\"ur Mathematik, Universit\"at Wien, Oskar-Morgenstern-Platz 1, 1090 Wien, Austria}
\email{martino.fassina@univie.ac.at}

\author{Yifei Pan}
\address{Department of Mathematical Sciences, Purdue University Fort Wayne, 2101 East Coliseum Boulevard,
Fort Wayne, IN 46805, USA}
\email{pan@pfw.edu}

\begin{document}

\maketitle

\begin{abstract}
Let $\Omega\subset \mathbb{R}^n, n\geq 2$, be an open set. For an elliptic differential operator $L$ on $\Omega$ with real analytic coefficients and a point $p\in\Omega$, we construct a smooth function $g$ with the following properties: $g$ is flat at $p$ and the equation $Lu=g$ has no smooth local solution $u$ that is flat at $p$.
\end{abstract} 

\begin{otherlanguage}{french} 
\begin{abstract}

Soit $ \Omega \subset \mathbb{R}^n, n\geq 2 $, un ensemble ouvert. Pour un op\'erateur diff\'erentiel elliptique $L$ sur $\Omega $ avec des coefficients analytiques r\'eels et un point $ p \in \Omega $, nous construisons une fonction lisse $g$ avec les propri\'et\'es suivantes: $g$ est plat en $p$ et l'\'equation $Lu=g$ n'a pas de solution locale lisse $u$ qui est plate en $p$.
\end{abstract}
\end{otherlanguage}


\section{Introduction}
Let $f$ be a smooth complex-valued function defined on an open set $\Omega\subseteq\R^n$. We say that $f$ is {\em flat} at a point $p\in\Omega$ if its $k$-jet vanishes at $p$ for all $k\in\mathbb{N}$. Functions with this property are ubiquitous in mathematics. For instance, any smooth function with compact support is flat at every point of the boundary of the support (Lemma \ref{simple}). Flat functions also play a role in the theory of PDEs,
 particularly in the study of the unique continuation property. In that context, the following question arises naturally: given a differential operator $L$ and the germ of a flat function $f$ at a point $p$, 
is there always a local solution $u$ to $Lu=f$ that is also flat at $p$? In this paper we show that the answer is negative for every elliptic operator $L$ with real analytic coefficients defined on an open set of $\mathbb{R}^n$, where $n\geq 2$. Here is our main result (see Theorem \ref{nonsolv}).
\begin{namedtheorem}[Main]\label{Theorem1}
Let $p\in\R^n$, $n\geq 2$, and let $L$ be an elliptic differential operator with real analytic coefficients defined on an open neighborhood $\Omega$ of $p$ in $\R^n$. There exists a germ of a flat function $g$ at $p$ with the property that there is no function $u$ flat at $p$ solving $Lu=g$. 
\end{namedtheorem}
Solvability here is meant in the sense of germs at $p$. The main theorem can also be restated in the following way: there exists a smooth germ $g$ vanishing to infinite order at $p$ such that every smooth local solution to $Lu=g$ vanishes to finite order at $p$. Note that, since $L$ is elliptic, it is well known that smooth local solutions always exist. 

We stress that this ``obstruction to solvability" on flat germs only occurs in dimension 2 or higher: in dimension $1$ local flat solutions for flat data always exist (Remark \ref{case1}). 

In a unpublished note \cite{CLPZ15}, the second author thogether with Zhihua Chen, Yang Liu, and Yuan Zhang, proved the same result for the special case of the Cauchy-Riemann operator $\bar\partial$ in $\C^n$. Their approach was based on an explicit integral formula for higher derivatives of solutions to the inhomogeneous Cauchy-Riemann equations. The methods of this paper, on the other hand, are totally different, and based on H\"ormander's treatment of the general theory of elliptic operators \cite{H}.

We believe that the phenomenon here presented will help shed light on several aspects of complex analysis. In particular, it is well understood that flat functions are related to the study of unique continuation \cite{C39, CP, Pan92, Pr, PW98}. In a future paper, we will show how the result stated above applies to the investigation of unique continuation for the $\bar\partial$ operator as well as to the study of minimal solutions for the $\bar\partial$ equation. Work in progress \cite{FP} has also revealed the critical consequences of this phenomenon on the local geometry of real hypersurfaces of finite type in complex dimension 2. 

The paper is essentially self-contained. Our proof of the main theorem relies on classical results in the theory of linear differential operators, for which we refer to \cite{H}. These results, together with the necessary background definitions, are recalled in Section \ref{Sec2}. There the proof of the main theorem is also carried out. The technical parts of the proof can be found in Sections \ref{theproof} and \ref{proofofProp}. In Sections \ref{complex} and \ref{Sec3} we present some applications of our result that fall within the framework of this paper.

\section{Local obstruction to solvability on flat germs}\label{Sec2}
Let $C^{\infty}(\Omega)$ be the ring of smooth complex-valued functions on an open set $\Omega\subset\mathbb{R}^n$ and $C^{\infty}_p$  
the corresponding ring of germs at a point $p\in\Omega$. Here $n\geq 2$ is a positive integer that will be fixed throughout the paper. Unless otherwise stated, all the functions considered will be complex-valued. Since our interest is local, we will usually take $p$ to be the origin and denote by $C^{\infty}_0$ the corresponding ring of smooth germs. Recall that elements of $C^{\infty}_0$ are smooth functions defined on some open neighborhood of $0$, and two smooth functions $f$ and $g$ coincide in $C^{\infty}_0$ if they agree on a neighborhood of $0$. The ring $C^{\infty}_0$ is local with unique maximal ideal $\mathfrak{m}=\{f\in C_0^{\infty},\, f(0)=0\}.$ Note, for $f\in C^{\infty}_0$, that
\begin{equation}\label{cf} 
f \text{ is flat at $0$} \Longleftrightarrow f\in \bigcap_{k=0}^{\infty}\mathfrak{m}^k.\end{equation} 

\begin{notation}
We denote by $c_f$ the ideal of germs of smooth functions that are flat at $0$. By \eqref{cf} we have
$$c_f=\bigcap_{k=0}^{\infty}\mathfrak{m}^k.$$
\end{notation}

Taking the Taylor series expansion at the origin defines a ring homomorphism 
\begin{equation}\label{maptilda}
\begin{split}
\sim\colon &C_0^{\infty}\longrightarrow \mathbb{C}[[x_1,\dots,x_n]]\\
&\,\,\,\,\,\,\,f\longrightarrow \widetilde{f}
\end{split}
\end{equation}
to the ring of formal power series $\mathbb{C}[[x_1,\dots,x_n]]$. The map in \eqref{maptilda} is clearly not injective: its kernel is precisely the ideal of flat germs $c_f$. In other words, $$f \in c_f\Longleftrightarrow \widetilde{f}=0.$$
It is a classical result of Borel \cite{B95} that the map in \eqref{maptilda} is surjective, that is, every formal power series in $\mathbb{C}[[x_1,\dots,x_n]]$ is the Taylor series at the origin of a smooth function $f$. We refer to \cite{Na85} for a modern proof  relying on the Whitney's Extension Theorem. At the risk of some redundancy, we state Borel's result in a separate lemma for future use.
\begin{lemma}(Borel)\label{Borel}
Let $h\in \mathbb{C}[[x_1,\dots,x_n]]$ be a formal power series. Then there exists a germ $f\in C^{\infty}_0$ whose Taylor expansion $\widetilde{f}$ at $0$ is such that
$\widetilde{f}=h$.
\end{lemma}

\begin{remark} Note that Lemma \ref{Borel} is not useful for constructing flat functions. Indeed, for the power series whose coefficients are all equal to $0$, Borel's construction yields the zero function.
\end{remark}
Let $C^{\omega}(\Omega)$ be the ring of complex-valued real analytic functions on an open set $\Omega\subset\mathbb{R}^n$ and $C^{\omega}_0$ the corresponding 
ring of germs at the origin. By definition,
$$f\in C_0^{\omega}\Longleftrightarrow \widetilde{f}(x)\text{ converges to } f(x)\text{ for } x \text{ in a neighborhood of }0.$$

We now introduce differential operators on an open subset $\Omega\subset \R^n$. 

\begin{notation}
We use the subscript notation for derivatives, writing $\partial_{x_j}$ in place of $\partial/\partial x_j$. For every multi-index $\alpha=(\alpha_1,\dots,\alpha_n)\in\mathbb{N}^n$, we let $D^{\alpha}:=\partial_{x_1}^{\alpha_1}\dots\partial_{x_n}^{\alpha_n}$. We write $|\alpha|$ for the length of $\alpha$, that is, $|\alpha|=\alpha_1+\dots+\alpha_n$. 
\end{notation}
 Recall that $L$ is a differential operator on $\Omega$ of order $m$ with real analytic coefficients if
\begin{equation}\label{L}
L=\sum_{|\alpha|\leq m}a_{\alpha}D^{\alpha},\quad a_{\alpha}\in C^{\omega}(\Omega).
\end{equation}


\begin{definition}\label{defin}
Let $L$ be a differential operator with real analytic coefficients on $\Omega \subset\R^n$ as in \eqref{L}. We say that $L$ is {\em elliptic} if $$\sum_{|\alpha|=m}a_{\alpha}(x)\xi^{\alpha}\neq 0 \text{\,\,\,for\,\,\,}x\in\Omega,\,\, \xi\in\mathbb{R}^n\setminus\{0\}.$$ 
\end{definition}

\begin{examplex} 
The Laplacian $\Delta=\sum_{j=1}^n\partial^2_{x_j}$ and its powers are elliptic operators in $\R^n$. The Cauchy-Riemann operator $\partial_{x}+i\partial_{y}$ is elliptic in $\R^2\simeq\C$.
\end{examplex}

The next lemma is the main technical point in the paper. The proof is presented in Section \ref{theproof} and requires Borel's Lemma \ref{Borel}. 
\begin{lemma}\label{main lemma}
 Let $L$ be an elliptic differential operator with real analytic coefficients defined on an open neighborhood of the origin in $\R^{n}$, where $n\geq 2$. Then there exist $G\in C_0^{\infty}$ and $g\in c_f$ such that $L(G)=g$ and the Taylor series expansion $\widetilde{G}$ of $G$ at 
$0$ is not convergent in any neighborhood of $0$.\
\end{lemma}

It is a deep classical result that elliptic differential operators with real analytic coefficients have analytic solutions for analytic data. This statement was first proved by Petrowsky \cite{Pe39} for homogeneous operators with constant coefficients. After a number of successive generalisations  \cite{MN57,M58a,M58b}, it is nowadays a textbook result. 
\begin{theorem}\cite[Theorem 7.5.1]{H}\label{Pet}
Let $L$ be an elliptic differential operator with real analytic coefficients defined on an open set $\Omega \subset\R^n$. If  $g\in C^{\omega}(\Omega)$ and $f$ is a distribution on $\Omega$ such that $Lf=g$ in $\Omega$ in the sense of distributions, then $f\in C^{\omega}(\Omega).$
\end{theorem}
We are now ready to state and prove our main theorem. The proof relies on Lemma \ref{main lemma} and Theorem \ref{Pet}. 
\begin{theorem}\label{nonsolv}
Let $p\in\R^n$, where $n\geq 2$, and let $L$ be an elliptic differential operator with real analytic coefficients defined on an open neighborhood $\Omega$ of $p$ in $\R^n$. Then there exists a germ of a flat function $g$ at $p$ such that there is no smooth germ $u$ flat at $p$ solving $Lu=g$. 
\end{theorem}
 
\begin{proof}
We can assume without loss of generality that $p$ is the origin. Let $g\in c_f$ and $G\in C^{\infty}_0$ be as in Lemma \ref{main lemma}. Assume by contradiction that there exists a flat solution  $u\in c_f$ to $Lu=g$. 
Then $\widetilde{G-u}=\widetilde{G}-\widetilde{u}=\widetilde{G}$. 
Recall that, by construction, the formal power series $\widetilde{G}$ does not converge in any neighborhood of $0$. Hence $G-u$ is not analytic at $0$. The function $G-u$, however, satisfies $L(G-u)=0$. Theorem \ref{Pet} therefore implies $G-u\in C^{\omega}_0$, which is a contradiction.
\end{proof}
 
 Here is a special instance of Theorem \ref{nonsolv} in the case of the Laplace operator and its powers in $\mathbb{R}^n$, for $n\geq 2$.
\begin{corollary}
For every positive integer $m$, there exists a germ of a smooth function $g$ that is flat at $0$ such that every local solution $u$ to $\Delta^m u=g$ vanishes to finite order at $0$.
\end{corollary} 

 \begin{remark}\label{case1}
The hypothesis $n\geq 2$ in Theorem \ref{nonsolv} cannot be dropped. Indeed, in dimension $1$, every elliptic differential operator $L$ with smooth coefficients admits local flat solutions for flat data. To see this fact, let $$L=\partial_x^n+a_{n-1}\partial_x^{n-1}+\dots+a_1\partial_{x}+a_0,\quad a_j\in C^{\infty}(\R)$$ and let $g\in c_f$ be the germ of a flat function at the origin. Let now $f$ be the (unique) smooth local solution at $0$ of the Cauchy problem
\begin{equation*}
\begin{cases}
Lf=g \\
\partial^j_x f (0)=0\quad j=0,1,\dots, n-1.
\end{cases}
\end{equation*}
We claim that $f\in c_f$. Note that
\begin{equation}\label{dim1}
\partial_x^n f=g-a_{n-1}\partial_x^{n-1}f-\dots-a_0 f.
\end{equation}
Evaluating \eqref{dim1} at $0$ we obtain $\partial_x^n f (0)=0$. Taking derivatives of \eqref{dim1} and exploiting the fact that $g$ is flat at $0$, one proves inductively that $\partial_x^{k}f(0)=0$ for all $k$. Hence $f\in c_f$.
 \end{remark}
 
 \begin{remark}
 We recall a famous example of local non-solvability due to Lewy \cite{L57}. See also \cite[pages 313-314]{SS11}. 
In $\R^3$ with variables $x,y,t$, consider the differential operator 
\begin{equation}\label{Lewy}
L=\frac{1}{2}\bigg{(}\frac{\partial}{\partial x}+i\frac{\partial}{\partial y}\bigg{)}-i(x+iy)\frac{\partial}{\partial t}.
\end{equation}
One can prove that there exists a function $f\in C^{\infty}(\R^3)$ which is flat at $0$ and such that the equation $Lu=f$ has no local solution at $0$ in the weak sense of distributions. Note that the operator $L$ defined in \eqref{Lewy} is not elliptic at $0$. Throughout the paper we only consider elliptic operators, and therefore local smooth solutions always exist.
 \end{remark}
 
 \section{Proof of Lemma \ref{main lemma}}\label{theproof}
 The proof of Lemma \ref{main lemma} relies on the following proposition, whose proof is in turn given in Section \ref{proofofProp}.
\begin{proposition}\label{prop}
Let $L$ be an elliptic differential operator of order $m$ with real analytic coefficients defined on an open neighborhood $\Omega$ of the origin in $\R^{n}$, where $n\geq 2$. Then there exist a sequence of polynomials $(p_k)_{k\in\mathbb{N}}\subset\C[x_1,\dots,x_{n-1}]$, a neighborhood $U$ of the origin in $\R^n$ and a sequence $(u_k)_{k\in\mathbb{N}}$ of real analytic functions in $U$ such that the following hold:
\begin{itemize}
\item For each $k$ we have $Lu_k=0$.
\item Each $u_k$ vanishes to order $k$ at the origin.
\item For each $k$, the power series expansion at $0$ of $u_k$ converges to $u_k$ on $U$.
\item For each $k$, the polynomial $p_k$ is homogeneous of degree $k$ and moreover $u_k=p_k$ on the hyperplane $x_n=0$.
\end{itemize}
\end{proposition} 
 
 \begin{proof}[Proof of Lemma \ref{main lemma}]
Let $U\subset\mathbb{R}^n, (p_k)_{k\in\mathbb{N}}\subset\C[x_1,\dots,x_{n-1}]$ and $(u_k)_{k\in \mathbb{N}}\subset C^{\omega}(U)$ be as in Proposition \ref{prop}. Let $Z(p_k)$ be the zero set of $p_k$ in $\R^{n-1}$. Since each $Z(p_k)$ has empty interior in $\R^{n-1}$, the Baire category theorem implies that the union $\bigcup_{k\in\mathbb{N}}Z(p_k)$ is nowhere dense in $\R^{n-1}.$ In particular, there exists a point $(\bar{x}_1,\dots,\bar{x}_{n-1})\in \mathbb{R}^{n-1}$ arbitrarily close to the origin of $\R^{n-1}$ such that $$p_k(\bar{x}_1,\dots,\bar{x}_{n-1})\neq 0 \text{ for all } k.$$ 
For each $k\in\mathbb{N}$ define 
\begin{equation*}
b_k:=\frac{k!}{|p_k(\bar{x}_1,\dots,\bar{x}_{n-1})|}
\end{equation*} and consider the formal power series $\sum_k b_k u_k\in\C[[x_1,\dots,x_{n}]]$. We first note that it is well defined. Indeed, for each integer $j$, only the functions $u_k$ with $k\leq j$ contain terms of order $j$. We claim that the formal power series $\sum_k b_k u_k$ does not converge in any neighborhood of 0. To prove this fact, consider the point $\bar{x}=(\bar{x}_1,\dots,\bar{x}_{n-1},0)$ and evaluate along the real line $\mathcal{L}:=\{x=t\bar{x} \,\vert \,t\in\mathbb{R}\}\subset\mathbb{R}^n$. Note that $\mathcal{L}$ is contained in the hyperplane $x_n=0$. We thus have 
\begin{equation*}\label{serie}
\sum_{k=0}^{\infty}|b_k u_k(t\bar{x})|=\sum_{k=0}^{\infty}|b_k p_k(t\bar{x}_1,\dots,t\bar{x}_{n-1})|=\sum_{k=0}^{\infty}|b_k t^kp_k(\bar{x}_1,\dots,\bar{x}_{n-1})|=\sum_{k=0}^{\infty}|t|^k k!,
\end{equation*}
which is clearly divergent for $t\in\R\setminus\{0\}$.
Borel's Lemma \ref{Borel} implies the existence of a smooth germ $G\in  C_0^{\infty}$ such that $\widetilde{G}=\sum_k b_kp_k(x)$. Let now $g:=L(G)$. It remains to prove that $g$ is flat at $0$. 
We thus need to show, for every multi-index $\alpha\in\mathbb{N}^n$, that $D^{\alpha}LG(0)=0$. By Taylor's theorem, there exists a neighborhood $U$ of the origin in $\mathbb{R}^n$ and a smooth function $\eta$ vanishing at $0$ to order at least $m+|\alpha|+1$ such that 
\begin{equation*}
G=\sum_{k=1}^{m+|\alpha|} b_ku_k+\eta\,\,\,\text{in } U.\end{equation*}
We then have
\begin{equation}\label{U}
D^{\alpha}LG=D^{\alpha}L\bigg{(}\sum_{k=1}^{m+|\alpha|} b_ku_k+\eta\bigg{)}=D^{\alpha}\bigg{(}\sum_{k=1}^{m+|\alpha|} b_k L u_k\bigg{)}+D^{\alpha}L\eta \,\,\,\text{ in }U.
 \end{equation}
Evaluating \eqref{U} at $0$, we obtain
 $$D^{\alpha}LG(0)=D^{\alpha}\bigg{(}\sum_{k=1}^{m+|\alpha|} b_k L u_k\bigg{)}(0)+D^{\alpha}L\eta(0)=0,$$
where we have exploited that $Lu_k=0$ near $0$ for each $k$ and that $\eta$ vanishes at $0$ to order at least $m+|\alpha|+1$. This concludes the proof.
\end{proof}
\section{Proof of Proposition \ref{prop}}\label{proofofProp}

We follow H\"ormander's reasoning in the proof of \cite[Theorem 5.11]{H}.
\begin{remark}
For the special case of $L$ homogeneous with constant coefficients, Proposition \ref{prop} also holds without the hypothesis of $L$ being elliptic (and the same is therefore true for Lemma \ref{main lemma}). In order to produce a sequence $(u_k)_{k\in\mathbb{N}}$ of solutions to $Lu=0$ with the property that each $u_k$ vanishes to order $k$ at the origin, one exploits the following observation. If $L=\sum_{|\alpha|=m}a_{\alpha}D^{\alpha}$ with $a_{\alpha}\in\mathbb{C}$, and $\zeta=(\zeta_1,\dots,\zeta_n)\in\C^n\setminus\{0\}$ is such that $\sum_{|\alpha|=m}a_{\alpha}{\zeta}^{\alpha}=0$, then $u_k:=(\zeta_1x_1+\dots+\zeta_nx_n)^k$ satisfies $Lu_k=0$. Hence the equation $Lu=0$ has homogeneous polynomial solutions of arbitrary degree. This was a commonplace 19th century observation that we learned from Bruce Reznick \cite[pag. 180]{Rz96}. 
\end{remark}
\begin{notation}
Throughout this section we denote by $\ord_0(f)$ the order of vanishing at the origin of a function $f$, by which we mean the smallest degree of a non-zero term in the Taylor series of $f$ at $0$.
\end{notation}

We start by choosing for each $k\in\mathbb{N}$ a homogeneous polynomial of degree $k$ with complex coefficients $p_k\in\C[x_1,\dots,x_{n-1}]$. We will build a solution to 
\begin{equation}\label{sist}
\begin{cases}
Lu=0 \\
u=p_k \text{ on } x_n=0,
\end{cases}
\end{equation}
and call this solution $u_k$.
\begin{itemize}
\item It will follow from our construction that $\ord_0(u_k)=k$.
\item We will keep track of the domain of definition of $u_k$ and prove that (if the polynomials $p_k$ are chosen appropriately) it is independent of $k$.
\end{itemize}

\begin{remark}
We will later see that not all sequences of polynomials $p_k$ serve our purpose, and we will need to require more conditions on them. 
\end{remark}

\begin{lemma}\label{lemma0}
Let $L$ be an elliptic operator of order $m$ defined on an open set $\Omega\subset\R^n$. Then, letting $\beta=(0,\dots,0,m)$, we can write $L$ as
\begin{equation}\label{canonicalform}
L=D^{\beta}-\sum_{\substack{|\alpha|\leq m\\ \alpha_n<m}} a_{\alpha}D^{\alpha}.
\end{equation}
\end{lemma}
\begin{proof}
Let $L=\sum_{|\gamma|=m}a_{\gamma}D^{\gamma}+\dots$, where dots stand for lower order terms. Let 
$$p(x,\xi):=\sum_{|\gamma|=m}a_{\gamma}(x)\xi^{\gamma},\quad x\in\Omega,\xi\in\R^n.$$
Note that $a_{\beta}(x)\neq 0$ for all $x\in\Omega$. Indeed, if $a_{\beta}(\bar{x})=0$ for some $\bar{x}\in\Omega$, then $p(\bar{x},(0,\dots,0,1))=0$, thus contradicting the ellipticity of $L$. Since $a_{\beta}$ is non-vanishing in $\Omega$, we can divide by $a_{\beta}$ and rearrange the terms to obtain \eqref{canonicalform}.
\end{proof}

Back to solving \eqref{sist}. We first make the substitution $v=u-p_k$.
We now have to solve 
\begin{equation}\label{sistchange}
\begin{cases}
L(v+p_k)=0 \\
v=0 \text{ on } x_n=0.
\end{cases}
\end{equation}
If we write $L$ as in Lemma \ref{lemma0}, the equation $L(v+p_k)=0$ becomes
\begin{equation}\label{eqv}
D^{\beta}v=\sum_{\substack{|\alpha|\leq m\\ \alpha_n<m}}a_{\alpha}D^{\alpha}v-Lp_k.
\end{equation}
Since $p_k$ is homogeneous of order $k$ and $L$ is an operator of order $m$, then $$\ord_0(Lp_k)\geq\max\{0,k-m\}.$$

\begin{notation}
We denote by $\mathbb{D}_R$ a polydisc of multi-radius $R=(R_1,\dots,R_n)$ centered at the origin, that is, $\mathbb{D}_R=\{(x_1,\dots,x_n)\in\R^n, |x_j|<R_j\}$.
\end{notation} 
Let $\mathbb{D}_R$ be a polydisc such that for each $\alpha$ the Taylor series of $a_{\alpha}$ converges in $\mathbb{D}_R$. It follows immediately that the function $-Lp_k$ also admits a representation as a convergent power series in $\mathbb{D}_R$, for each $k$ and regardless of the choice of $p_k$.

The next lemma shows how we can solve \eqref{sistchange} by recursion.
\begin{lemma}\label{lemmanu}
Fix $k\in\mathbb{N}$. There exists a sequence of functions $\{v_{\nu}\}_{\nu\in\mathbb{N}}$ such that 
\begin{equation}\label{recurs}
\begin{cases}
D^{\beta}v_{\nu+1}=\sum_{\alpha}a_{\alpha}D^{\alpha}v_{\nu}-Lp_k \\
v_{\nu}=0 \text{ on } x_n=0.
\end{cases}
\end{equation}
Moreover, the following properties hold: \begin{itemize}
\item For each $\nu\in\mathbb{N}$, the function $v_{\nu}$ is a convergent power series in $\mathbb{D}_R.$
\item For $\nu\geq 1$, we have $\ord_0(v_{\nu})\geq k$.
\end{itemize}
\end{lemma}
\begin{proof}
Let $\nu_{0}=0$ on $\mathbb{D}_R$. We look for solutions of 
\begin{equation}\label{nuzero}
\begin{cases}
D^{\beta}v_1=-Lp_k \\
v_1=0 \text{ on } x_n=0.
\end{cases}
\end{equation}
If $-Lp_k(x)=\sum_{\gamma}c_{\gamma}x^{\gamma}$ in $\mathbb{D}_R$, then we choose the solution to \eqref{nuzero} given by
\begin{equation}\label{nuuno}
v_1(x):=\sum_{\gamma}c_{\gamma}x^{\gamma+\beta}\frac{\gamma_n!}{(\gamma_n+m)!}.
\end{equation}
The series in \eqref{nuuno} is still convergent in $\mathbb{D}_R$. Moreover, $\ord_0(Lp_k)\geq\max\{0,k-m\}$, and therefore $\ord_0(v_1)\geq k$.

We now proceed by induction on $\nu$. Assume that the lemma is proved for $\nu$. We want $v_{\nu+1}$ to solve \eqref{recurs}. Let $\sum_{\gamma}c_{\gamma}x^{\gamma}$ be an expression for $\sum_{\alpha}a_{\alpha}D^{\alpha}v_{\nu}-Lp_k$ as a convergent power series in $\mathbb{D}_R$. Then let $$v_{\nu+1}(x)=\sum_{\gamma}c_{\gamma}x^{\gamma+\beta}\frac{\gamma_n!}{(\gamma_n+m)!}.$$ We see that $\ord_0(v_{\nu+1})\geq m+\max\{\ord_{0}(v_{\nu})-m, \ord_0(Lp_k)\}$. Applying the inductive hypothesis, we conclude that $\ord_0(v_{\nu+1})\geq k$.
\end{proof}
\begin{remark}
The sequence $v_{\nu}$ should really be called $v^k_{\nu}$, since there is one sequence for each $k$, or better for every polynomial $p_k$. Note that we have not yet put any restriction on the choice of polynomials. We will use the notation $v^k_{\nu}$ whenever we want to emphasize the dependence on $p_k$.
\end{remark}
With the sequence $v_{\nu}$ from Lemma \ref{lemmanu} available, the proof of Proposition \ref{prop} follows easily from the next lemma.

\begin{lemma}\label{tecn}
There exists an open neighborhood $U$ of $0$ in $\R^n$ and a choice of polynomials $p_k$ such that for each $k$ the sequence $v_{\nu}^k$ converges uniformly on $U$. 
\end{lemma}
\begin{proof}[Proof of Proposition \ref{prop}]
For each $k\in\mathbb{N}$, let $v_k:=\lim_{\nu\to\infty}v^k_{\nu}$. By Lemma \ref{tecn} and Lemma \ref{lemmanu}, each $v_k$ is a uniform limit of analytic functions in $U$. Hence $v_k$ is itself analytic in $U$ and moreover, for every multi-index $\alpha\in\mathbb{N}^n$, we have $$D^{\alpha}v_{k}=\lim_{\nu\to\infty}D^{\alpha}v^k_{\nu}.$$
If we then take $\lim_{\nu\to\infty}$ on both sides of the following equation
$$D^{\beta}v^k_{\nu+1}=\sum_{\alpha}a_{\alpha}D^{\alpha}v^k_{\nu}-Lp_k,$$
we see that $v_k$ solves \eqref{sistchange}. Tracing back our substitution, we let $u_k:=p_k+v_k$. Note that each $u_k$ is analytic in $U$ and solves \eqref{sist}. In particular, $Lu_k=0$, as required.
Now note that, for each multi-index $\alpha$ of length $|\alpha|<k$, we have
\begin{equation}\label{qend}
D^{\alpha}v_{k}(0)=\lim_{\nu\to\infty}D^{\alpha}v^k_{\nu}(0)=0.
\end{equation}
The last equality in \eqref{qend} follows from $\ord_0(v^k_{\nu})\geq k$ (Lemma \ref{lemmanu}). We conclude that each $v_k$ vanishes to order at least $k$ at $0$. Now consider $u_k=p_k+v_k$. Recall that $p_k$ is a homogeneous polynomial of degree $k$ in the variables $x_1,\dots,x_{n-1}$. By construction, $v_k=0$ for $x_n=0$. In particular, the Taylor series of $v_k$ at $0$ does not have any term not involving the variable $x_n$. Hence in the sum $v_k+p_k$ no cancellation occurs and $u_k$ vanishes to order exactly $k$ at $0$, as wanted. \end{proof}

The only thing left to prove is Lemma \ref{tecn}. We first state and prove two simple auxiliary lemmas from \cite{H}.

\begin{lemma}\label{l1}
Let $R>0$ and $g$ a $C^1$ function in one real variable $x$ defined for $|x|<R$. Assume that
$$|g'(x)|\leq |x|^a,\quad |x|<R, \,\,\,\text{ and }\,\,\, g(0)=0,$$
where $a\geq 0$. It follows that 
$$|g(x)|\leq\frac{ |x|^{a+1}}{a+1},\quad |x|<R.$$
\end{lemma}
\begin{proof}
The statement follows immediately from $g(x)=\int_0^{x}g'(t)dt.$
\end{proof}
\begin{lemma}\label{l2}
Assume that $g(x)=\sum_ja_jx^j, a_j\in\C$ is a convergent power series in $|x|<R$ and $$|g(x)|\leq (R-|x|)^{-a},\quad |x|<R,$$
where $a\geq 0$. It follows that
$$|g'(x)|\leq e(1+a)(R-|x|)^{-a-1},\quad |x|<R.$$
\end{lemma}
\begin{proof}
Complexify $g$ to a holomorphic function in $|\zeta|<R$. Let $0<\epsilon<\rho=R-|\zeta|$ and let $|\zeta_1-\zeta|<\epsilon$. Adding $\rho$ on both sides of $-\epsilon\leq |\zeta|-|\zeta_1|,$ we obtain $$\rho-\epsilon \leq \rho+|\zeta|-|\zeta_1|=R-|\zeta_1|.$$ Hence 
\begin{equation}\label{rho}
(R-|\zeta_1|)^{-a}\leq (\rho-\epsilon)^{-a}.
\end{equation}
Combining \eqref{rho} and the hypothesis, we have
$$|g(\zeta_1)|\leq (R-|\zeta_1|)^{-a} \leq (\rho-\epsilon)^{-a}.$$
By Cauchy's formula, we obtain
$$|g'(\zeta)|\leq (\rho-\epsilon)^{-a}\epsilon^{-1}.$$
Letting $\epsilon=\rho/(1+a)$, we get
$$|g'(\zeta)|\leq (1+a)(1+a^{-1})^{a}\rho^{-a-1}\leq e(1+a)\rho^{-a-1}.$$
\end{proof}

For each $k\in\mathbb{N}$, define the sequence of differences $w^k_{\nu}:=v^k_{\nu+1}-v^k_{\nu}$. Since the $v^k_{\nu}$ satisfy \eqref{recurs}, then for every $k$ we have
\begin{equation}\label{w}
D^{\beta}w^k_{\nu+1}=\sum_{\alpha}D^{\alpha}w^k_{\nu},\quad \nu=0,1,\dots
\end{equation}

\begin{remark}\label{remmy} Note that Lemma \ref{tecn} is proved if we can show that there exists a neighborhood $U$ of $0$ in $\mathbb{R}^n$ and a choice of polynomials $p_k$ such that for all $k$ the series $\sum_{\nu} |w^k_{\nu}(x)|$ converges uniformly in $U$.  
\end{remark}

Recall that $\mathbb{D}_R$ is a polydisc centered at $0$ of multi-radius $R=(R_1,\dots,R_n)$ such that all the functions $a_{\alpha}$ admit a representation as convergent power series in $\mathbb{D}_R$. From now on we assume, without loss of generality, that $R_j<1$ for all $j$.
\begin{notation}
We denote by $D^j_n$ the derivative $D^{\gamma}$, where $\gamma=(0,\dots,0,j)$.
\end{notation}

\begin{lemma}\label{lemmalungo}
There exists a constant $C$ and a choice of the polynomials $p_k$ such that the following estimate holds for every $k$:
\begin{equation}\label{estimate}
|D^{\beta}w^k_{\nu}(x)|\leq C^{\nu+1}|x_n|^{\nu} d(x)^{-m\nu-1},\quad x\in \mathbb{D}_R,\quad \nu=0,1,\dots
\end{equation}
Here $d(x)=\prod_{j=1}^{n-1}(R_j-|x_j|).$
\end{lemma}
\begin{proof}
We will prove the lemma by induction on $\nu$. 

We choose the polynomials $p_k$ so that $|Lp_k|\leq 1$ in $\mathbb{D}_R$ for all $k$. This is easily achieved by replacing each $p_k$ by $p_k/M$, where $M=\max\{|Lp_k(x)|,x\in\overline{\mathbb{D}}_R\}$. 

Recall that $w^k_0=v_1^k-v_0^k$, $v_0^k\equiv 0$ and $D^{\beta}v_1^k=-Lp_k.$ Hence 
\begin{equation}\label{almostdone}
|D^{\beta}w^k_0(x)|=|Lp_k(x)|\leq 1,\quad x\in\mathbb{D}_R.
\end{equation}
Note that, for each $i$, $$ \frac{R_i}{R_i-|x_i|}\geq 1,\quad x\in\mathbb{D}_R.$$
Letting $C:=\prod_{i=1}^{n-1} R_i$, we then have $Cd(x)^{-1}\geq 1.$ Combining with \eqref{almostdone}, we get
$$|D^{\beta}w^k_0(x)|\leq C d(x)^{-1},\quad x\in \mathbb{D}_R, \quad k\in\mathbb{N},$$
that is, \eqref{estimate} holds for $\nu=0$.

We now want to prove that $C$ can be chosen sufficiently large so that \eqref{estimate} can be proved by recursion. Assume that \eqref{estimate} holds for $\nu$ and for every $k\in\mathbb{N}$. Repeated application of Lemma \ref{l1} gives, for $0\leq j<m$ and $k\in\mathbb{N}$,
\begin{equation}\label{stimaderivatan}
|D^{j}_n w^k_{\nu}(x)|\leq C^{\nu+1}|x_n|^{\nu+(m-j)}d(x)^{-m\nu-1}\Bigg(\prod_{i=0}^{m-j-1}\frac{1}{\nu+i}\Bigg),\quad x\in\mathbb{D}_R.
\end{equation}
Note that $|x_n|<1$ and 
$$\frac{1}{\nu+i}\leq \frac{1}{\nu},\quad \text{for } \nu\in\mathbb{N},\quad  i=0,\dots,m-j-1.$$
Hence \eqref{stimaderivatan} becomes
\begin{equation}\label{stimameglion}
|D^{j}_n w^k_{\nu}(x)|\leq C^{\nu+1}|x_n|^{\nu+1}d(x)^{-m\nu-1}\nu^{j-m},\quad x\in\mathbb{D}_R.
\end{equation}
To conclude the inductive step, we need to prove an estimate on $|D^{\beta}w^k_{\nu+1}(x)|$ for $x\in\mathbb{D}_R$. In virtue of \eqref{w}, it is enough to have an estimate on each $|D^{\alpha}w^k_{\nu}(x)|$, where $|\alpha|\leq m$, $\alpha_n<m$. Assume that $\alpha_n=j<m.$ A repeated application of Lemma \ref{l2} to \eqref{stimameglion} yields, for every $k\in\mathbb{N}$,
\begin{equation}\label{stimader}
|D^{\alpha}w^k_{\nu}(x)|\leq C^{\nu+1}|x_n|^{\nu+1}\nu^{j-m}e^{|\alpha|-j}d(x)^{-m\nu-1-|\alpha|}\Bigg(\prod_{i=1}^{|\alpha|-j}(m\nu+1+i)\Bigg),\quad x\in\mathbb{D}_R.
\end{equation}
Here we have used that $(R_i-|x_i|)<1$ for all $i$, and therefore $$(R_i-|x_i|)^{-a}<(R_i-|x_i|)^{-a-1} \,\,\,\text{for}\,\,\, a\geq 0.$$ 
Equation \eqref{stimader} implies, for $x\in\mathbb{D}_R$ and $k\in\mathbb{N}$,
\begin{equation}\label{superestimate}
\begin{split}
|D^{\alpha}w^k_{\nu}(x)|&\leq C^{\nu+1}|x_n|^{\nu+1}\nu^{j-m}e^{m-j}d(x)^{-m\nu-1-(m-j)}(m\nu+m+1)^{m-j}\\
&\leq C^{\nu+1}|x_n|^{\nu+1}e^{m-j}d(x)^{-m(\nu+1)-1}\bigg(m+\frac{m}{\nu}+\frac{1}{\nu}\bigg)^{m-j}\\
&\leq C^{\nu+1}|x_n|^{\nu+1}e^{m}d(x)^{-m(\nu+1)-1}(2m+1)^m.
\end{split}
\end{equation}
Let $A$ be a constant such that $\sum_{\alpha}|a_{\alpha}|\leq A$ in $\mathbb{D}_R$. Recalling that
\begin{equation*}
D^{\beta}w^k_{\nu+1}=\sum_{\alpha}D^{\alpha}w^k_{\nu}
\end{equation*} 
and combining with \eqref{superestimate}, we see that \eqref{estimate} holds for $\nu+1$ provided that $C$ is  chosen large enough so that $C\geq Ae^m(2m+1)^m$.
\end{proof}

\begin{remark}
In Lemma \ref{lemmalungo} we could have employed any sequence of homogeneous polynomials $p_k$ for which there exists a constant $M$ such that $|Lp_k|\leq M$ in $\mathbb{D}_R$ for all $k$.
\end{remark}

\begin{proof}[Proof of Lemma \ref{tecn}]
As noted in Remark \ref{remmy}, it is enough to prove that there exists a neighborhood $U$ of $0$ in $\mathbb{R}^n$ and a choice of polynomials $p_k$ such that for all $k$ the series $\sum_{\nu} |w^k_{\nu}(x)|$ converges uniformly in $U$. We exploit the estimate proved in Lemma \ref{lemmalungo}. Consider the equation \eqref{stimameglion} for $j=0$. Let $U$ be the neighborhood of $0$ where $C|x_n|/d(x)^m<1.$ Then the series $\sum_{\nu} |w^k_{\nu}(x)|$ converges uniformly in $U$ for every $k\in\mathbb{N}$.
\end{proof} 
  \section{A local obstruction for $\bar\partial$ on complex manifolds}\label{complex}
Let $n\geq 1$, and consider the space $\R^{2n}$ with coordinates $\{x_1,\dots,x_n,y_1,\dots,y_n\}$. We identify $\R^{2n}$ with $\C^n$ in the usual way, by letting $z_j=x_j+iy_j$. Recall the definition of the derivatives $\partial_{\bar{z}_j}:=\partial_{x_j}+i\partial_{y_j}$ and $\partial_{z_{j}}:=\partial_{x_j}-i\partial_{y_j}$. The Cauchy-Riemann operator $\bar\partial$ is defined on smooth functions $f$ in $\C^n$ by $$\bar\partial f:=\sum_{j=1}^n \partial_{\bar{z}_j}f \,d\bar{z}_j.$$ The following is an immediate consequence of Theorem \ref{nonsolv}.
\begin{corollary}\label{c1}
There exists a germ of a smooth function $g$ in $\C$ that is flat at $0$ and such that every local solution $u$ to $\bar\partial u=g\,d\bar{z}$ vanishes to finite order at 0.
\end{corollary} 
\begin{notation}
Recall that for a smooth function $f$ we denote by $\widetilde{f}$ its Taylor series at $0$.
\end{notation}
Corollary \ref{c1} can be easily generalized to several complex variables as follows.
 \begin{corollary}\label{c2}
For every $n$, there exists a $\bar\partial$-closed smooth $(0,1)$ form $\varphi=\sum_{j=1}^n\varphi_j\,d\bar{z}_j$ in $\mathbb{C}^n$ defined in a neighborhood of $0$ such that:
\begin{itemize}
\item The functions $\varphi_j$ are all flat at $0$.
\item Every local solution to $\bar\partial u=\varphi$ vanishes to finite order at $0$.
\end{itemize}
 \end{corollary}
 \begin{proof}
For every $j\in\{1,\dots,n\}$, let $\varphi_j=g(z_j)$, where $g$ is as in Corollary \ref{c1}. Note that every solution of $\bar\partial u=\varphi$ is of the form $u=u_1(z_1)+\dots+u_n(z_n)+h(z_1,\dots,z_n)$, where $h$ is holomorphic and each smooth function $u_j(z_j)$ is a solution to $\bar\partial u_j(z_j)=g(z_j)d\bar{z}_j$. By our choice of $g$, every $u_j$ vanishes to finite order at $0$. We now want to conclude that $u$ itself vanishes to finite order at $0$. Assume by contradiction that the Taylor series $\widetilde{u}$ of $u$ at $0$ is identically zero. Hence \begin{equation}\label{Taylorsomma}
0=\widetilde{u}=\widetilde{u_1}+\dots+\widetilde{u_n}+\widetilde{h}.
\end{equation}
Evaluating \eqref{Taylorsomma} at $(z_1,0,\dots,0)$ we obtain 
\begin{equation}\label{Tayloreq}
\widetilde{u_1}(z_1)+\widetilde{h}(z_1,0,\dots,0)=0.
\end{equation} Since $u_1$ vanishes to finite order at $0$ and $h$ is holomorphic at $0$, then \eqref{Tayloreq} implies that $u_1$ is also holomorphic at $0$. This is absurd, since $\partial_{\bar{z}_1}u_1=g(z_1)$, and $g\not\equiv 0$.
 \end{proof}
  
\begin{notation} In this section we denote by $c_f$ the ring of germs of smooth functions in $\C^n$ that are flat at $0$. Moreover, we write $c_{f}^{0,1}$ for the complex vector space of germs of smooth $(0,1)$ forms whose components are flat at $0$. 
\end{notation}

Corollary \ref{c2} shows that there exist closed forms $\varphi\in c^{0,1}_f$ such that there is no function $u\in c_f$ with $\bar\partial u=\varphi$. The aim of the rest of this section is to give a characterization for such elements $\varphi$ presenting the ``flat non-solvability" property. To this end, we introduce below the concept of formally holomorphic functions. We then prove in Theorem \ref{charact} that $\bar\partial u=\varphi$ is non-solvable in the flat category if and only if $\varphi$ is the image under $\bar\partial$ of a formally holomorphic function.
  
  \begin{notation}
  When $z=(z_1,\dots,z_n)\in\C^n$ and $\alpha=(\alpha_1,\dots,\alpha_n)$ is a multi-index, we write $z^{\alpha}$ for the product $\prod_{j=1}^nz_j^{\alpha_j}$ and $\partial_{z}^{\alpha}$ for $\partial_{z_1}^{\alpha_1}\partial_{z_2}^{\alpha_2}\dots \partial_{z_n}^{\alpha_n}$.
  \end{notation}
  
\begin{definition}
We say that a smooth function $f$ on $\C^n$ is {\em formally holomorphic} at a point $p\in\C^n$ if there exist constants $c_{\alpha}\in\mathbb{C}$ such that the Taylor series of $f$ at $p$ is given by
\begin{equation}\label{Tay}
\sum_{\alpha}c_{\alpha}(z-p)^{\alpha}.
\end{equation}
In other words, in the Taylor series expansion of $f$ at $p$ the barred variables $\bar{z}_j$ do not appear.
\end{definition}

\begin{remark} Borel's Lemma \ref{Borel} ensures that for each choice of coefficients $c_{\alpha}$ there exists a formally holomorphic function $f$ at $p$ whose Taylor series at $p$ is given by \eqref{Tay}. Note that if \eqref{Tay} converges in a neighborhood of $p$, then $f$ is holomorphic at $p$. 
\end{remark}

As usual, since our interest is local, we will assume that $p$ is the origin. We denote by $\mathcal{H}_0$ the subring of $C^{\infty}_0$ consisting of the germs of smooth functions that are formally holomorphic at $0$.
\begin{lemma}
If $f$ is formally holomorphic at $0$, then the components of $\bar\partial f$ are flat at $0$. In other words, $f\in\mathcal{H}_0$ implies $\bar\partial f\in c_{f}^{0,1}$.
\end{lemma}
\begin{proof}
Consider $\partial_{z}^{\gamma}\partial_{\bar{z}}^{\beta}$ with $|\gamma|+|\beta|=m$. Since $f$ is formally holomorphic, we can write $$f=\sum_{|\alpha|\leq m+1}a_{\alpha}z^{\alpha}+\eta,$$ where $\eta$ is a smooth function vanishing at $0$ to order at least $m+2$. Then $\bar\partial f=\bar\partial\eta$, that is, $\partial_{\bar{z}_j}f=\partial_{\bar{z}_j}\eta$ for $j=1,\dots,n$. All these identities are intended in the ring of germs $C^{\infty}_0$. We thus have, for each $j$, $$\partial_{z}^{\gamma}\partial_{\bar{z}}^{\beta}(\partial_{\bar{z}_j}f)(0)=\partial_{z}^{\gamma}\partial_{\bar{z}}^{\beta}(\partial_{\bar{z}_j}\eta)(0)=0,$$ where the last equality follows from the hypothesis on the order of vanishing of $\eta$. Since the same argument can be repeated for any derivative $\partial_{z}^{\gamma}\partial_{\bar{z}}^{\beta}$, we conclude that each function $\partial_{\bar{z}_j}f$ is flat at $0$, and therefore $\bar\partial f\in c_{f}^{0,1}$.
\end{proof}

The next theorem is the main result of this section.
\begin{theorem}\label{charact}
Let $\varphi\in c_{f}^{0,1}$ be a $\bar\partial$-closed $(0,1)$ form, $\varphi\not\equiv 0$. The following are equivalent:
\begin{enumerate}
\item There exists no element $u\in c_f$ such that $\bar\partial u=\varphi$.
\item $\varphi\in\bar{\partial}\mathcal{H}_0.$
\end{enumerate} 
\end{theorem}
\begin{proof}

First assume that $(1)$ holds and let $f\in C^{\infty}_0$ be a smooth local solution to $\bar\partial f=\varphi$. Note that such a solution always exists by the $\bar\partial$-Poincar\'e Lemma. By $(1)$ we have that $f$ vanishes to finite order at $0$. We want to prove that $f\in\mathcal{H}_0$. Assume by contradiction that the Taylor series expansion of $f$ at $0$ contains a term of the form $a_{\beta\gamma}z^{\beta}\bar{z}^{\gamma}$, where $|\gamma|\geq 1$ and $a_{\beta\gamma}\in\C\setminus\{0\}$. It follows that there exists a derivative $\partial^{\beta}_{z}\partial^{\delta}_{\bar{z}}$ with $|\delta|=|\beta|-1$ and an index $j\in\{1,\dots,n\}$ such that 
\begin{equation}\label{flattt}
\partial^{\beta}_z\partial^{\delta}_{\bar{z}}(\partial_{\bar{z}_{j}} f)(0)\neq 0.
\end{equation} 
Recall that $\bar\partial f=\varphi$ and that all the components of the form $\varphi$ are flat at $0$. Hence \eqref{flattt} leads to a contradiction.

Conversely, assume that $(2)$ holds, that is, $\varphi=\bar\partial f$ for some $f\in\mathcal{H}_0$. Since $\varphi\not\equiv 0$, then $f$ is not holomorphic at $0$, and its Taylor series $\widetilde{f}$ is divergent in any neighborhood of $0$. Assume by contradiction that there exists $u\in c_f$ such that $\bar\partial u=g$. Then $\bar\partial(f-u)=0$, so that $f-u$ is holomorphic at $0$. In particular, the Taylor series $\widetilde{f-u}$ converges in a neighborhood of $0$. This is a contradiction, since $u\in c_f$ implies $\widetilde{f-u}=\widetilde{f}$.  
 \end{proof}

 \begin{figure}[h]
\centering
 \begin{tikzpicture}[mycircle/.style={draw,circle,minimum size=3.8cm}]
\node at (-4.5,-1.4) {$\bar\partial$}; 
\node at (0,-2.3) {$cc_f^{0,1}$}; 
\node at (-4.5,4.7) {$\bar\partial$}; 
\node at (-0.1,-1) {$\bar\partial c_f$}; 
\node at (0,0.9) {$\bar\partial \mathcal{H}_0$}; 
\node at (-10,2.7) {$\mathcal{O}_0$}; 
\node at (-10,-2.3) {$c_f$}; 
\node at (-10,5) {$\mathcal{H}_0$}; 
\draw[black] (-10,0)circle (1.9cm); 
\draw[black,fill=gray, fill opacity = 0.2] (-10,5)circle (1.9cm); 
\tkzDefPoint(-10,1.9){E};
\node at (-10,1.55) {$0$};
\draw[fill] (E) circle [radius=0.1]  ;
\draw[-{Latex[length=3mm]}](-8,5) to[out=0,in=130] (-0.5,1.2);
\draw[-{Latex[length=3mm]}](-8,-1) to (-0.75,-1);
\draw [fill=gray, fill opacity = 0.2] let \n1={-10+1.9*cos(-72}, \n2={5+1.9*sin(-72)}, \n3={-10+1.9*cos(252}, \n4={5+1.9*sin(252)}  in
(E)--(\n1,\n2) ([shift={(-72:1.9)}]-10,5) arc (-72:-108:1.9cm)--(E);
      \node[mycircle,black](hdm){};
      \draw[black,fill=gray, fill opacity=0.2, looseness=1.5] (-1.9,0) to[out=70,in =110] (0,0) to [out=-70,in=-110](1.9,0) arc (0:180:1.9cm);
     
   \end{tikzpicture} 
   \caption{The failure of the $\bar\partial$-Poincar\'e lemma for flat germs.} \label{fig1}
 \end{figure}
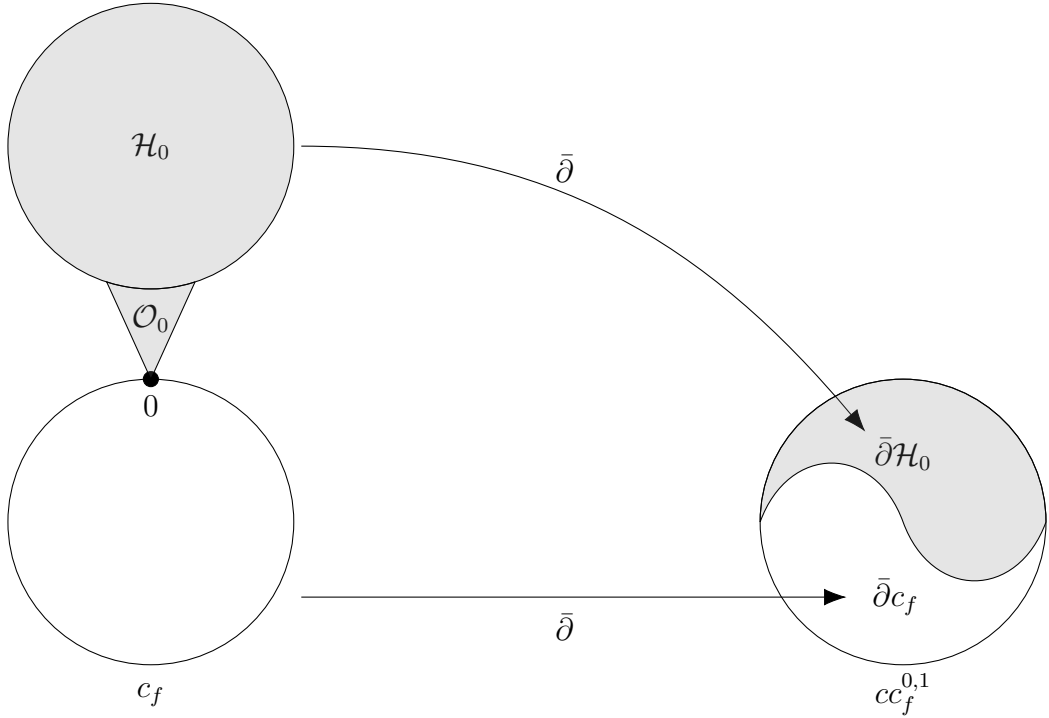

 Figure \ref{fig1} provides a simplified illustration of the phenomenon discussed above. The two discs at the bottom of the picture represent respectively the ring $c_f$ of flat germs in $\C^n$ and the complex vector space $cc_f^{0,1}$ of all $\bar\partial$-closed, smooth $(0,1)$ forms with flat components in $\C^n$. Theorem \ref{charact} shows that the map 
 \begin{equation}\label{dibarnonsurj}
 \bar\partial\colon c_f\to cc_f^{0,1}
 \end{equation}
  is not surjective. The same proposition also shows that the elements of $cc_f^{0,1}$ without a preimage in $c_f$ are precisely those in the set $\bar\partial\mathcal{H}_0$. In other words, the map \eqref{dibarnonsurj} becomes surjective if we enlarge the domain to include the ring of germs of formally harmonic functions $\mathcal{H}_0$. Hence $$\bar\partial \colon c_f\cup \mathcal{H}_0\to cc_f^{0,1}$$ is surjective. Note that the ring $\mathcal{H}_0$ corresponds to the shaded region on the left side of Figure \ref{fig1}. The intersection $c_f\cap \mathcal{H}_0$ consists of the zero function only. The ring $\mathcal{H}_0$ contains as a subring the collection of holomorphic germs, denoted as $\mathcal{O}_0$. For every $g\in \mathcal{O}_0$ we have $\bar\partial g=0$, where $0$ is the $(0,1)$ form whose components are identically zero. In Figure \ref{fig1} such form is identified with the curve between the sets $\bar\partial\mathcal{H}_0$ and $\bar\partial c_f$.
 
 \begin{remark} All the statements of this section are local, and therefore hold for the $\bar\partial$ operator on a general complex manifold.
 \end{remark}
 
 We close this section with some remarks on the real case. Let $c_f$ denote the ring of germs of flat functions at $0$ on $\R^n$, $n\geq 1$. Let $cc_f^1$ be the complex vector space of germs at $0$ of closed smooth one forms with flat components on $\R^n$. The ``flat non-solvability" phenomenon does not occur for the real differential \begin{equation}\label{realdiff}
 d\colon c_f\to cc_f^1.
 \end{equation} That is, the map \eqref{realdiff} is always surjective. Indeed, let $f=\sum_j f_jdx_j$ be a closed one form with flat components $f_j\in c_f$. Let $U$ be an open convex neighborhood of $0$ where all the functions $f_j$ are defined. By the standard Poincar\'e Lemma, there exists a function $u$ defined in $U$ such that $u(0)=0$ and
\begin{equation}
\partial_{x_j} u=f_j \quad j=1,\dots, n.
\end{equation}
 It is easily seen that such $u$ is flat at $0$, that is, $u\in c_f$. 
 
As a last remark, we point out that the difference in behavior between the real and the complex case is rooted in the fact that the operator $\bar\partial$, unlike $d$, has an infinite dimensional nullspace. 
 
 \section{Further Consequences}\label{Sec3}
 \subsection{Corollaries to Theorem \ref{nonsolv}}Let $L$ be an elliptic differential operator with real analytic coefficients defined on some open neighborhood $\Omega$ of $0$ in $\R^n$, where $n\geq 2$. Here we present some easy consequences of Theorem \ref{nonsolv}.
 \begin{corollary}\label{cinj}
The following hold:
\begin{enumerate}
\item The linear map $L\colon c_f\rightarrow c_f$ is injective but not surjective.
\item The induced map on the quotient spaces $L\colon C^{\infty}_0/c_{f}\rightarrow C^{\infty}_0/c_f$ 
is surjective but not injective.
\end{enumerate}
\end{corollary}
 \begin{proof}
{\em (1)} Assume that $u\in c_f$ is such that $Lu=0$. By Theorem \ref{Pet}, $u$ is real analytic at $0$, and hence $u\equiv 0$ near $0$. This proves that $L\colon c_f\rightarrow c_f$ is injective. The failure of surjectivity was proved in Theorem \ref{nonsolv}. 
\newline 
{\em (2)} Surjectivity follows from the classical existence theory \cite{Ni}. To show that injectivity fails, it is enough to consider $g\in c_f$ as in Theorem \ref{nonsolv}. In fact, for every $u\in C^{\infty}_0$ such that $Lu=g$, we have $u\notin c_f$. 
 \end{proof}
 \begin{corollary}
For every integer $N$ there exists a smooth germ $u\notin c_f$ such that $L^ku\notin c_f$ for $k=0,\dots,N$, but $L^{N+1}u\in c_f$.

\end{corollary}  
 \begin{proof}
 Let $u\in C^{\infty}_0$ be such that $L^{N+1}u=g$, where $g\in c_f$ is as in Theorem \ref{nonsolv}.  
\end{proof} 
\begin{corollary}\label{cin}  
  For each $k\in\mathbb{N}$ there is a strict inclusion $L^{k+1} c_f\subset L^{k} c_f$.
  \end{corollary}
  
  \begin{proof} For $k\in\mathbb{N}$, let $h=L^kg$, where $g\in c_f$ is as in Theorem \ref{nonsolv}. Assume by contradiction that $h\in L^{k+1}c_f$. Then there exists $u\in c_f$ such that $L^{k+1}u=h$. Since $L$ is injective on $c_f$, this implies that $Lu=g$, which is absurd by the choice of $g$. Hence $h\notin L^{k+1}c_f$, and the inclusion $L^{k+1}c_f\subset L^kc_f$ is strict. 
  \end{proof}
By Corollary \ref{cin} there is a strictly decreasing chain of $\C$-vector spaces
$$c_f\supset L c_f\supset L^2 c_f\supset L^3 c_f\supset \dots $$
 
 We now consider the $L$-invariant subspace $K$ of $c_f$ defined by $K:=\bigcap_{k\in\mathbb{N}}L^kc_f$. 
 \begin{proposition}
 Either $K=\{0\}$ or $K$ is an infinite dimensional $\C$-vector space.
 \end{proposition}
 \begin{proof}
 Assume that there exists a non-zero element $g\in K$. Then for every positive integer $k$ there exists an element $f_k\in c_f$ such that $L^kf_k=g$. It is easy to see that $f_k\in K$ for all $k$. Since $L\colon c_f\rightarrow c_f$ is injective (Corollary \ref{cinj}), then $f_k\not\equiv 0$ for all $k$. We now claim that for each $m>0$ the set $\{f_1,\dots,f_m\}$ is linearly independent. By contradiction, assume that $$f_1=c_2f_2+\dots+c_mf_m,\quad c_j\in\C.$$
 Then, recalling that $Lf_{k+1}=f_k$ for each $k$, we have 
 
 \begin{equation}\label{matrix}
 \begin{bmatrix}
L & 0 & \dots & 0\\
0 & L & \dots & 0\\
0 & . & . & 0\\
.&.&.&.\\
0&\dots&0&L\\
\end{bmatrix} \begin{bmatrix}
f_2\\
f_3 \\
. \\
.\\
f_m\\
\end{bmatrix}=\begin{bmatrix}
c_2 & c_3 & \dots & c_m\\
1 & 0 & \dots & 0\\
0 & 1 & \dots & 0\\
.&.&.&.\\
0&\dots&1&0\\
\end{bmatrix}\begin{bmatrix}
f_2\\
f_3 \\
. \\
.\\
f_m\\
\end{bmatrix}.
\end{equation}

Letting $A$ be the matrix of constants that appears on the right side of \eqref{matrix} and $I_{m-1}$ the identity matrix of size $m-1$, we have that the vector $f:=(f_2,f_3,\dots,f_m)$ satisfies the elliptic system $(L I_{m-1}-A)f=0$. Recall that elliptic systems with real analytic coefficients have analytic solutions for analytic data \cite{MN57}. Hence the functions $f_k$ are analytic at $0$. Since $f_k\in c_f$ for all $k$, that is, all the $f_k$ are flat at $0$, then $f_2=f_3=\dots= f_m=0$, which is a contradiction.
 \end{proof}
\begin{remark}
 It would be interesting to find an example of an operator $L$ for which the space $K$ is infinite dimensional.
 \end{remark}
 \subsection{Solutions with compact support}
 We now present an application of the ideas introduced in Section \ref{Sec2} to the study of compactly supported solutions of elliptic operators with real analytic coefficients. Recall that the support of a function $f$, which we denote as $\supp(f)$, is the closure of the set where $f$ is non-zero. We write $\partial\, \supp (f)$ for the boundary of the support. The connection with flat functions is made clear in the next simple lemma.
\begin{lemma}\label{simple}
If $f\in C^{\infty}(\R^n)$ has compact support and $x_0$ is a point of the boundary of the support of $f$, then $f$ is flat at $x_0$. 
\end{lemma}
\begin{proof}
Consider a sequence of points $x_j\in\mathbb{R}^n\setminus \supp(u)$ such that $x_j\rightarrow x_0$ and observe, for every multi-index $\alpha$, that $D^\alpha f(x_0)= \lim_{x_j\to x_0} D^\alpha f(x_j)=0$. 
\end{proof}
 

\begin{proposition}\label{supporto}
Let $L$ be an elliptic differential operator with real analytic coefficients defined on an open set $\Omega\subset\R^n$. Let $f$ be a smooth compactly supported function on $\Omega$ and $u$ a smooth compactly supported solution to $Lu=f$. Then 
\begin{equation}\label{relationonsupport}
\partial\,\supp (u)\subseteq \supp (f) \subseteq \supp (u).
\end{equation}
\end{proposition}
\begin{proof}
The inclusion $\supp (f) \subseteq \supp (u)$ is clear. For the inclusion $\partial\, \supp (u)\subseteq \supp (f)$, consider a point $x_0\in \partial\, \supp (u)$. Since $u$ is flat at $x_0$ (Lemma \ref{simple}) but not identically zero in any neighborhood of $x_0$, then $u$ is not analytic at $x_0$. Assume by contradiction that $x_0\notin \supp (f)$. Then $f\equiv 0$ in a neighborhood of $x_0$. Hence $u$ satisfies $Lu=0$ in the sense of germs at $x_0$. By Theorem \ref{Pet}, $u$ is analytic at $x_0$, which is a contradiction. This concludes the proof of \eqref{relationonsupport}. 
\end{proof}

\begin{remark}The topological relations described in Proposition \ref{supporto} between the support of the initial datum and the support of the solution are well known in the case of the Cauchy-Riemann operator $\bar\partial$ (see for example \cite[Proposition 1.1]{LTS}). 
\end{remark}
\begin{remark} For $n=2$, if $\supp(f)$ is simply connected, then \eqref{relationonsupport} gives $\supp(u)=\supp(f)$. The operator $\bar\partial$ on $\R^2\simeq\C$ shows that in general $\supp (u)\neq \supp (f)$ if $\supp(f)$ is not simply connected. Consider the following example. Let $\varphi\colon\R\rightarrow\R$ be a smooth non-negative  function such that $\supp(\varphi)\subset[1,2]$. Let $f\colon\C\rightarrow\C$ be defined as $f(z)=z\varphi(|z|^2)$. Then $f$ is smooth in $\C$, $\supp(f)\subseteq \{1\leq |z|\leq 2\}$ and $\int_{\C}\zeta^{-1}f(\zeta)\, d\bar\zeta\wedge d\zeta\neq 0$. A solution $u$ for $\bar\partial u=f(z)dz$ is given by $$u(z)=-\frac{1}{2\pi i}\int_{\C}\frac{f(\zeta)}{\zeta-z}\, d\bar\zeta\wedge d\zeta.$$
With this choice of $f$ we have that $u(0)\neq 0$, and therefore $\supp(u)\neq\supp(f)$.
\end{remark}

\section{Acknowledgements}
The first author acknowledges useful conversations on the topics of this paper with Luca Baracco, Christine Laurent-Thi\'ebaut, and Bruce Reznick. He also wishes to thank Ethan Addison for his careful reading of the manuscript, and the Department of Mathematics of Purdue University - Fort Wayne for the warm hospitality and financial support.


\begin{thebibliography}{WWW}


\bibitem[B95]{B95} Borel, \'E. Sur quelques points de la th\'eorie des fonctions. {\em Ann. Sci. \'Ecole Norm. Sup. (3)} {\bf 12} (1895), 9--55.
\bibitem[C39]{C39} Carleman, T. Sur un probl\`eme d'unicit\'e pur les syst\`emes d'\'equations aux d\'eriv\'ees partielles \`a deux variables ind\'ependantes. {\em Ark. Mat. Astr. Fys.} {\bf 26} (17) (1939).
\bibitem[CLPZ15]{CLPZ15} Zhihua Chen, Yang Liu, Yifei Pan and Yuan Zhang, Examples of flat solutions to the Cauchy-Riemann equations. Unpublished note (2015).
\bibitem[CP12]{CP} Coffman, A., Pan, Y. Smooth counterexamples to strong unique continuation for a Beltrami system in $\mathbb{C}^2$. {\em Comm. Partial Differential Equations} {\bf 37} (12) (2012), 2228--2244.
\bibitem[FP]{FP} Fassina, M., Pan, Y. Remarks on the global distribution of the points of finite D'Angelo type. In preparation.
\bibitem[H63]{H} H\"ormander, L. {\em Linear partial differential operators.} Die Grundlehren der mathematischen Wissenschaften, Bd. 116 Academic Press, Inc. Publishers, New York; Springer-Verlag, Berlin-G\"ottingen-Heidelberg, 1963.
\bibitem[LS18]{LTS} Laurent-Thi\'ebaut, C., Shaw, M.-C. Solving $\bar\partial$ with prescribed support on Hartogs triangles in $\mathbb{C}^2$ and $\mathbb{CP}^2$. {\em Trans. Amer. Math. Soc.} {\bf 371} (2019), no. 9, 6531--6546.
\bibitem[L57]{L57}Lewy, H. An example of a smooth linear partial differential equation without solution. 
{\em Ann. of Math. (2)} {\bf 66} (1957), 155--158. 
\bibitem[MN57]{MN57} Morrey Jr., C. B., Nirenberg, L. On the analyticity of the solutions of linear elliptic systems of partial differential equations. {\em Comm. Pure Appl. Math.} {\bf 10} (1957), 271--290.
\bibitem[M58a]{M58a} Morrey Jr., C. B. On the analyticity of the solutions of analytic non-linear elliptic systems of partial differential equations. I. Analyticity in the interior. {\em Amer. J. Math.} {\bf 80} (1958), 198--218. 
\bibitem[M58b]{M58b} Morrey Jr., C. B. On the analyticity of the solutions of analytic non-linear elliptic systems of partial differential equations. II. Analyticity at the boundary. {\em Amer. J. Math.} {\bf 80} (1958), 219--237. 
\bibitem[N85]{Na85} Narasimhan, R. {\em Analysis on real and complex manifolds.} North-Holland Mathematical Library, 35. North-Holland Publishing Co., Amsterdam, 1985.
\bibitem[Ni11]{Ni} Nirenberg, L. On Elliptic Partial Differential Equations. In: Faedo S. (eds) Il principio di minimo e sue applicazioni alle equazioni funzionali. C.I.M.E. Summer Schools, vol 17. Springer, Berlin, Heidelberg, 2011.
\bibitem[P92]{Pan92} Pan, Y. Unique continuation for Schr\"odinger operators with singular potentials. {\em Comm. Partial Differential Equations} {\bf 17} (1992), 953--965. 
\bibitem[PW98]{PW98} Pan, Y., Wolff, T. A remark on unique continuation. {\em J. Geom. Anal.} {\bf 8} (4) (1998), 599--604.
\bibitem[P39]{Pe39} Petrowsky, I. G. Sur l'analyticit\'e des solutions des syst\`emes d'\'equations diff\'erentielles,  {\em Rec. Math. N. S. [Mat. Sbornik]} {\bf 47} (5) (1939), 3--70.
\bibitem[Pr60]{Pr} Protter, M. H. Unique continuation for elliptic equations. {\em Trans. Amer. Math. Soc.}  {\bf 95} (1960), 81--91.
\bibitem[R96]{Rz96} Reznick, B. Homogeneous polynomial solutions to constant coefficient PDE's. {\em Adv. Math.} {\bf 117} (2) (1996), 179--192. 
\bibitem[SS11]{SS11} Stein, E. M., Shakarchi, R. Functional analysis. Introduction to further topics in analysis. Princeton Lectures in Analysis, 4. {\em Princeton University Press, Princeton, NJ,} 2011.

\end{thebibliography}
\end{document}